\documentclass[10pt]{amsart}
\usepackage{amssymb}
\usepackage{color}
\usepackage{graphicx}
\usepackage{float}
\usepackage[all,cmtip]{xy}
\newcommand{\SA}{{\mathcal{A}}}

\newcommand{\SC}{{\mathcal{C}}}

\newcommand{\SE}{{\mathcal{E}}}
\newcommand{\SF}{{\mathcal{F}}}

\newcommand{\SH}{{\mathcal{H}}}

\newcommand{\SL}{{\mathcal{L}}}

\newcommand{\SP}{{\mathcal{P}}}

\newcommand{\ST}{{\mathcal{T}}}

\newcommand{\Z}{\mathbb{Z}}

\newcommand{\R}{\mathbb{R}}

\renewcommand{\S}{\mathbb{S}}
\renewcommand{\SC}{\mathcal{C}}
\renewcommand{\SA}{\mathcal{A}}
\newcommand{\T}{{\mathbb{T}}}

\newcommand{\im}{\operatorname{im}}

\newcommand{\End}{\operatorname{End}}

\newcommand{\Cont}{\operatorname{Cont}}

\newcommand{\Diff}{\operatorname{Diff}}
\newcommand{\Op}{\operatorname{Op}}

\newtheorem{proposition}{Proposition}
\newtheorem{theorem}[proposition]{Theorem}
\newtheorem{definition}[proposition]{Definition}
\newtheorem{lemma}[proposition]{Lemma}

\newtheorem{corollary}[proposition]{Corollary}
\newtheorem{remark}[proposition]{Remark}

\newtheorem{example}[proposition]{Example}

\begin{document}

\title{h--principle for 4--dimensional Contact Foliations}

\subjclass[2010]{Primary: 53D10.}
\date{March, 2014}

\keywords{contact structures, foliations, h--principle.}

\author{Roger Casals}
\address{Universidad Aut\'onoma de Madrid and Instituto de Ciencias Matem\'aticas -- CSIC.
C. Nicol\'as Cabrera, 13--15, 28049, Madrid, Spain.}
\email{casals.roger@icmat.es}

\author{\'Alvaro del Pino}
\address{Universidad Aut\'onoma de Madrid and Instituto de Ciencias Matem\'aticas -- CSIC.
C. Nicol\'as Cabrera, 13--15, 28049, Madrid, Spain.}
\email{alvaro.delpino@icmat.es}

\author{Francisco Presas}
\address{Instituto de Ciencias Matem\'aticas -- CSIC.
C. Nicol\'as Cabrera, 13--15, 28049, Madrid, Spain.}
\email{fpresas@icmat.es}

\begin{abstract}
In this article we introduce the topological study of codimension--1 foliations which admit contact or symplectic structures on the leaves. A parametric existence h--principle for foliated contact structures is provided for any cooriented foliation in a closed oriented 4--fold.
\end{abstract}
\maketitle

In the present article we study the contact and symplectic topology of the leaves of a codimension--1 foliation in a given manifold. We introduce the definitions and objects of interest, detail their basic properties and explain related constructions. In this sense, the article establishes the foundations of foliated contact and symplectic topology. This work also proves a parametric existence h--principle for foliated contact 4--folds.\\

Symplectic and contact structures give rise to geometries in which there is a balance between flexible and rigid phenomena. On one hand, existence h--principles for both structures are satisfied in the open case \cite{EM,Gr2} and in the closed case there exists a flexible class of closed contact manifolds \cite{El,BEM}. On the other hand, the extension of a germ of a symplectic structure to the interior of a ball can be obstructed \cite{Gr1}. From the viewpoint of their transformation groups, relevant geometric structures have been found \cite{EKP,Po} and rigid dynamical behaviours are known to occur, for instance Arnold's Conjecture in the symplectic case \cite{FO,LT} and Weinstein's Conjecture in the contact case \cite{Ta,We}.\\

These features distinguish contact and symplectic topology in the realm of differential topology. We expect that contact and (strong) symplectic foliations, introduced in this article, will form an interesting subclass of foliations. In particular, there should also exist a dichotomy between flexible and rigid results. The existence of a foliated contact structure for a codimension--1 foliation is a flexibility result. This article exhibits such flexibility in the case of foliated 4--folds.\\

Let us state the main result. For this, we set up the following notation. The pair $(V, \SF)$ denotes a smooth oriented manifold $V$ endowed with a regular cooriented codimension--1 foliation $\SF$. (This pair is fixed along the article.) Note that a manifold $V$ admits such a foliation if and only if $\chi(V)=0$, confer \cite[Theorem 1.A]{Th}. Consider the space $\SP(V, \SF)$ of codimension--2 distributions $\xi$ on $V$ such that $\xi\subset T\SF$, endowed with the compact--open topology. Define the following two spaces:
\begin{eqnarray*}
\SC(V, \SF)& = & \{\xi \in\SP(V, \SF):\xi\mbox{ induces a contact structure on each leaf of $\SF$}\}, \\
\SA(V, \SF)& = & \{(\xi, J): \xi \in \SP(V, \SF), J\in\End(\xi), J^2= -id\mbox{ and $J$ compatible with } \xi\}.
\end{eqnarray*}
The elements in $\SC(V, \SF)$ are called contact foliations (or foliated contact structures), and those in $\SA(V, \SF)$ almost contact foliations (or foliated almost contact structures).\\

The space of compatible almost complex structures for a fixed contact structure is contractible, hence there exist maps $\pi_k\iota: \pi_k\SC(V,\SF) \to \pi_k\SA(V,\SF)$ between the homotopy groups. The study of $\SA(V,\SF)$ lies within algebraic topology, whereas the understanding of the space of contact structures $\SC(V,\SF)$ requires geometry.\\

In $3$--dimensional contact topology there are no (geometric) obstructions for the existence of a contact structure in a given 3--fold \cite{Ma,Lu}. The main result of this article states that the obstructions for the existence of a codimension--1 contact foliation are also strictly topological:

\begin{theorem}\label{thm:main}
Let $(\xi,J)$ be a foliated almost contact structure on $(V^4,\SF)$. There exists a homotopy $\{(\xi_t, J_t)\} \subset \SA(V, \SF)$ of foliated almost contact structures such that $(\xi_0, J_0)=(\xi, J)$ and $\xi_1$ is a contact foliation.
\end{theorem}
The proof of the Theorem also implies a (weak version of the) parametric h--principle:
\begin{corollary}\label{cor:main}
The maps $\pi_k\iota: \pi_k\SC(V, \SF)\longrightarrow\pi_k\SA(V, \SF)$ are surjective.
\end{corollary}
Theorem \ref{thm:main} is the case $k=0$ of the parametric result stated in Corollary \ref{cor:main}. To be precise, Corollary \ref{cor:main} is not deduced from the statement of Theorem \ref{thm:main} but rather from its method of proof. This is an instance of a flexible phenomenon in foliated contact topology. The argument of Theorem \ref{thm:main} certainly uses the classification of overtwisted structures \cite{El}, but new ideas coming from foliation theory are needed . In particular, the relative parametric version of a foliated h--principle does not a priori hold (this is implied by Proposition \ref{prop:notinjective}).\\

We are also able to present strictly geometric situations, as in the following example. Consider a codimension--1 foliation $(N^3,\SF)$ in a closed oriented 3--fold and the circle bundle $\pi:\mathbb{S}(\SF)\longrightarrow N$ corresponding to the Euler class $e(T\SF)\in H^2(N,\Z)\cong[N,BU(1)]$.
\begin{proposition} \label{prop:notinjective}
The map $\pi_0\iota: \pi_0\SC(\mathbb{S}(\SF), \pi^*\SF) \to \pi_0\SA(\mathbb{S}(\SF),\pi^* \SF)$ is not injective.
\end{proposition}
This Proposition is a direct application of the tight--overtwisted dichotomy in $3$--dimensional contact topology. It also turns foliated contact structures into an interesting object from the geometric viewpoint.

\begin{remark}
There is currently no definition of a flexible --overtwisted-- class of foliated contact structures such that, for instance, Theorem \ref{thm:main} establishes an isomorphism when restricted to this class.\\

There are examples of contact foliations with tight and overtwisted leaves (and the existence of a transverse family of overtwisted disks would require tautness of $\SF$).
\end{remark}
Rigidity results in contact topology can be generalized to foliated contact topology. For instance, a foliated Weinstein conjecture can be stated and we may apply the techniques used in the contact case to prove them. In particular, the foliated Weinstein conjecture holds for a particular class of foliated contact structures, confer \cite{PP}.\\

The article is organized as follows. Section \ref{sec:intro} defines the objects of interest. It also contains the foliated counterparts of several classical theorems in contact and symplectic topology. Section \ref{sec:contsymptaut} describes different procedures for constructing foliated contact (and symplectic) structures. Although Sections \ref{sec:intro} and \ref{sec:contsymptaut} are not entirely required for the proof of Theorem \ref{thm:main}, they establish the foundations of the theory and provide the reader with intuition on the results and methods that work for these foliated geometries.\\

Section \ref{sec:hprin} states the $h$--principles used in the proofs (coming from contact geometry) and technical lemmas which will be used in the proof of Theorem \ref{thm:main}. Section \ref{sec:proof} establishes Theorem \ref{thm:main} in the case in which the foliation $\SF$ is taut. Section \ref{sec:general} adapts the techniques developed in Section \ref{sec:proof} to complete the proof of Theorem \ref{thm:main} in the general case and also explains the proof of Corollary \ref{cor:main}.   \\

{\bf Acknowledgements:} The authors are thankful to V. Ginzburg, K. Niederkr\"uger, D. Pancholi and A. Rechtman for useful discussions. The authors are supported by the Spanish National Research Project MTM2010--17389. The present work is part of the authors activities within CAST, a Research Network Program of the European Science Foundation. The article was partially written during the stay of the first author at Stanford University, he is grateful to Y. Eliashberg for his hospitality. 


\section{Foliated Contact and Symplectic Topology}\label{sec:intro}
Let $V$ be a smooth oriented manifold. The foliations $\SF$ considered on $V$ are smooth regular codimension one, cooriented and oriented foliations. The distributions appearing in the article are also assumed to be cooriented and oriented.  Given a leaf $\SL$ of a foliation $\SF$, the natural inclusion of the leaf in the foliation is denoted by $\iota_\SL:\SL \longrightarrow V$.\\

In this Section we introduce the objects of interest and provide proofs of their basic properties. Subsection \ref{ssec:contfol} is dedicated to contact foliations and Subsection \ref{ssec:symfol} to symplectic foliations. Constructions and examples shall be explained in Section \ref{sec:contsymptaut}.

\subsection{Contact foliations}\label{ssec:contfol}

The central objects in this article are contact foliations. This notion generalizes the concept of a contact fibration and intertwines the contact topology of the leaves with the their global behaviour governed by the foliation.

\begin{definition}\label{def:contfol} A contact foliation $(\SF,\xi)$ on $V$ is a foliation $\SF$ on $V$ and a codimension--2 distribution $\xi\subset T\SF$ such that $(\SL, i^*_\SL\xi)$ is a contact manifold for every leaf $\SL$. 
\end{definition}
The manifold $V$ must be even dimensional for a contact foliation $(\SF,\xi)$ to exist. Given a codimension--2 distribution $\xi\subset T\SF$, an extension of $\xi$ is a codimension--1 distribution $\Theta\subset TV$ such that $\xi=\Theta\cap T\SF$. The reader should compare this with a contact fibration \cite{CP,Le}. A pair $(\SF,\Theta)$ such that $\Theta\cap\SF$ is a contact structure on the leaves will be referred to as an extended contact foliation. A pair $(\beta,\alpha)$ of $1$--forms is \textsl{associated} to $(\SF,\Theta)$ if $\ker\beta=T\SF$ and $\ker\alpha=\Theta$.

\begin{definition}\label{def:contcon} Let $(\SF,\Theta)$ be an extension of a contact foliation $(\SF,\xi)$ with associated pair $(\beta,\alpha)$. The associated contact connection is the distribution $\SH_{\Theta}=\xi^{\perp d\alpha}\subset (\Theta,d\alpha)$.
\end{definition}

The contact connection is a line field contained in $\Theta$ and satisfying $\xi \oplus \SH_\Theta = \Theta$. In particular, $T\SF \oplus \SH_\Theta = TV$. Observe that $d\alpha$ does not necessarily vanish on $\SH_\Theta$: rather, contraction with a vector field in $\SH_\Theta$ yields a form that vanishes on $\xi$ and $\SH_\Theta$, so it must be a multiple of $\alpha$. Note that the contact connection does not depend on the choice of $\alpha$, it only depends on the extended distribution $\Theta$. Contact connections have been used before in the case of contact fibrations, see for instance \cite{CP}.

\begin{example}\label{ex:mapping} Let $(\SL,\xi_\SL = \ker\alpha_\SL)$ be a contact manifold. The manifold $\SL \times [0,1]$, with coordinates $(p,t)$, has a natural contact foliation structure given by
$$ \tilde\SF = \SL\times\{t\}, \quad  \tilde\xi(p,t)=(\iota_{\SL\times\{t\}})_*\xi_{\SL}(p). $$
Consider $\phi\in\Cont(\SL,\xi_\SL)$ and the associated mapping torus $M(\phi)$ (this is a contact fibration). From the contact foliation viewpoint, it inherits the contact foliation structure $(\SF, \xi)$ from $(\SL \times [0,1], \tilde\SF, \tilde\xi)$ obtained as a quotient. Given a vector field $X\subset T\SF$ such that $\phi_*X=X$, an extension $(\SF,\Theta)$ is obtained by declaring
$$ \Theta = \xi\oplus\langle\partial_t+X\rangle. $$
Denote $H=\alpha_\SL(X)$. Then the contact connection $\SH_\Theta$ is the distribution generated by the vector field $\langle\partial_t+\widetilde X\rangle$ satisfying the equations
$$ \alpha_\SL(\widetilde X)-H = 0,\quad d\alpha_\SL(\widetilde X,v) + dH(v)=0 \quad \forall v\in\ker\alpha_\SL. $$
Hence $\widetilde X$ is the Hamiltonian vector field associated to $H$.

\begin{figure}[ht]
\includegraphics[scale=0.3]{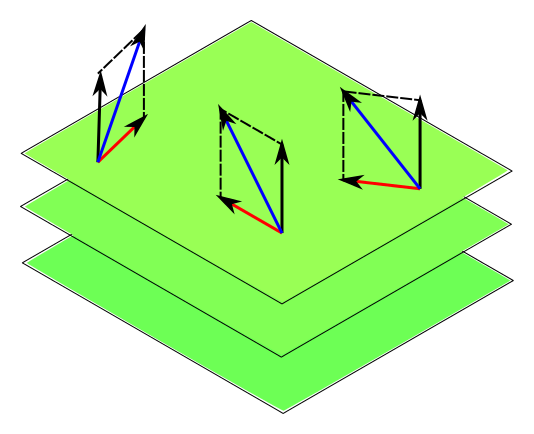}
\caption{Foliation $\SF$ in green and the vector fields $\partial_t$ (black), $X$ (red), and $\partial_t+X$ (blue).}\label{fig:example}
\end{figure}
\end{example}

Given an extension $(\SF,\Theta)$ defined by $(\beta,\alpha)$, we associate two vector fields $T$ and $R$ uniquely determined by the following equations:
$$\alpha(T)=0,\quad (i_Td\alpha)\wedge\alpha=0,\quad\beta(T)=1$$
$$\alpha(R)=1,\quad( i_Rd\alpha)\wedge\beta=0,\quad\beta(R)=0.$$
These vector fields $T$ and $R$ will be referred to as the transverse field and the Reeb field of $(\beta,\alpha)$. Note that $\SH_\Theta = \langle T \rangle$. \\

The parallel transport between the leaves in Example \ref{ex:mapping} is by contactomorphisms. This is a particular instance of the following
\begin{lemma}\label{lem:contparallel} Let $(\SF,\Theta)$ be an extended contact foliation and $(\beta,\alpha)$ an associated pair. Then
$$ \SL_T\alpha=d\alpha(T,R)\alpha.$$
In particular the distribution $\Theta$ is preserved by the flow of the transverse field $T$ of $(\beta,\alpha)$.
\end{lemma}
It does not hold in general that $T$ preserves $\SF$ since that would imply $d\beta=0$.
\begin{proof}
The Cartan formula gives $\SL_T\alpha=di_T\alpha+i_Td\alpha=i_Td\alpha$. The statement follows from the fact that $i_Td\alpha=d\alpha(T,R)\alpha$, which can be readily verified by evaluation in each factor of $\xi\oplus R\oplus T$.
\end{proof}
Consider the space
$\SE(\SF,\xi)=\{(\SF,\Theta): (\SF,\Theta)\mbox{ is an extension of }(\SF,\xi)\}$. The space of connection $1$--forms in a smooth vector bundle has a natural affine structure. This is also the case for $\SE(\SF,\xi)$. Indeed, one can prove the following
\begin{lemma} \label{lem:affine}
$\SE(\SF,\xi)$ has an affine structure.
\end{lemma}
\begin{proof}
Consider an extension $(\SF,\Theta_0)$ with associated pair $(\beta_0, \alpha_0)$. The space
$$ A = \{\alpha \in \Omega^1(V) : \alpha \wedge \beta_0 = \alpha_0 \wedge \beta_0\}$$
is an affine space modelled on the space of $1$--forms vanishing on $\SF$. The map $\Psi: A \longrightarrow \SE(\SF,\xi)$ defined by $\Psi(\alpha)=(\SF,\ker\alpha)$ is bijective and endows $\SE(\SF,\xi)$ with an affine structure. \\
\end{proof}
In particular, the space $\SE(\SF,\xi)$ is contractible and the choice of extension for a foliated contact structure is unique up to homotopy.\\

The previous definitions and facts suffice for the reader to follow the proof of Theorem \ref{thm:main}. However, we consider appropriate to include a short discussion on the transformations appearing in foliated contact topology. A diffeomorphism of $V$ that preserves both $\SF$ and $\xi$ will be called a foliated contactomorphism. The infinitesimal symmetries are described as follows.
\begin{definition}
Let $(V, \SF, \xi)$ be a contact foliation. The space of contact vector fields of $(V, \SF, \xi)$ is defined as the subspace of $\mathfrak{X}(V)$ of those vector fields that preserve $\SF$ and $\xi$. 
\end{definition}

If an extension $(\SF, \Theta)$ is fixed and an associated pair $(\beta, \alpha)$ is given, the space of contact vector fields is equivalently defined as:
\begin{eqnarray*} \label{eq:contactvf}
\mathcal{C}(\xi) = \{ X \in \mathfrak{X}(V) : & \SL_X (\alpha \wedge \beta) = f\alpha \wedge \beta & \text{ for some } f \in C^\infty(V) \\
																							& \SL_X \beta = g \beta & \text{ for some } g \in C^\infty(V) \}.
\end{eqnarray*}
A flow by foliated contactomorphisms is induced by a 1--parametric family of contact vector fields. 
\begin{lemma}
The space $\mathcal{C}(\xi)$ is a Lie algebra. It contains a distinguished ideal $\mathcal{C}(\xi) \cap \mathfrak{X}(\SF)$. 
\end{lemma}
\begin{proof}
The first statement follows since the space of vector fields preserving a distribution is a Lie algebra. The second claim is immediate, since $[X, Y] \in T\SF$, for $X$ preserving $\SF$ and $Y \in \SF$.
\end{proof}

This incursion into foliated contact vector fields and the use of extensions of a contact foliation lead to a proof of the foliated version of Gray's stability in contact topology (i.e. the moduli space of contact structures is discrete). The precise statement reads as follows.

\begin{lemma}[Foliated Gray's Stability]\label{lem:gray}
Let $\SF$ be a codimension $1$--foliation on a closed manifold $V$ and consider a family $\{\xi_t\}_{t \in [0,1]}$ of codimension--2 distributions such that $(\SF, \xi_t)$ is a foliated contact structure for every $t \in [0,1]$. Then there exists a global flow $\{\phi_t\}_{t\in [0,1]} \in\Diff(V)$ tangent to $\SF$ such that $\phi_t^* \xi_t =\xi_0$.
\end{lemma}
\begin{proof}
We consider a smooth $1$--parametric family of extensions $\Theta_t$ and their associated pairs $(\beta, \alpha_t)$. We require a flow $\phi_t$ tangent to the leaves (and therefore preserving $\SF$) such that $\phi_t^* \xi_t =\xi_0$. In terms of the associated forms this reads as
\[ \phi_t^* (\alpha_t) \wedge \beta = g_t\alpha_0 \wedge \beta \]
for a suitable choice of $\{g_t\}_t \in C^\infty(V)$, since $\phi_t$ preserves $\beta$ up to scaling. We now apply the foliated version of Moser's argument.\\

Denote by $X_t$ the vector field generating the required flow $\phi_t$ (that is, $X_t \circ \phi_t = \dot{\phi}_t$) and we further suppose that $X_t$ is contained in the contact structures $\xi_t$. Differentiating the above condition with respect to $t$ we obtain:
\begin{equation} \label{eq:gray1}
 \phi_t^*(\SL_{X_t} \alpha_t + \dot\alpha_t) \wedge \beta = g_t'\alpha_0 \wedge \beta = \dfrac{g_t'}{g_t} (\phi_t^* \alpha_t) \wedge \beta.
 \end{equation}
Define $\lambda_t = (\phi_t)_* \dfrac{g_t'}{g_t}$ and $(\phi_t)_* \beta = m_t \beta$. Equation (\ref{eq:gray1}) implies
$$(\SL_{X_t} \alpha_t + \dot\alpha_t) \wedge (m_t \beta) = \lambda_t m_t\alpha_t \wedge \beta.$$
Since $m_t$ is strictly positive, the equation can we written as
\[ (\SL_{X_t} \alpha_t + \dot\alpha_t) \wedge \beta = \lambda_t \alpha_t \wedge \beta, \]
\begin{equation} \label{eq:gray2}
 (i_{X_t} d\alpha_t + \dot\alpha_t) \wedge \beta = \lambda_t \alpha_t \wedge \beta. 
\end{equation}
This is an equation in $1$--forms. In particular, it has to be satisfied by the Reeb vector $R_t$, thus yielding the condition
\[ (i_{R_t} \dot{\alpha_t}) \beta = \lambda_t \beta. \]
This reads $\phi_t^*(i_{R_t} \dot\alpha_t) = (\ln g_t)'$, which is an ODE with an unique solution once the initial condition $g_0=1$ is fixed. Now, since $R_t \in \ker(\dot\alpha_t - \lambda_t \alpha_t)$, Equation (\ref{eq:gray2}) can solved uniquely for $X_t \in \xi_t$.
\end{proof}

There are foliated analogues of several concepts and results in contact topology. We have introduced the essential notions, and the reader should be able to define any further objects and prove foliated versions of basic results. See for instance Lemmas \ref{lem:moser} and \ref{lem:Darboux}.\\

In the next subsection we shall discuss in brief foliated symplectic topology. There is a subtlety added in comparison to the contact case, yet experts familiar with symplectic and Hamiltonian fibrations will (hopefully) find the definitions quite reasonable. \\

\subsection{Symplectic foliations}\label{ssec:symfol}

In this subsection we treat the case of (strong) symplectic foliations. It is not required in order to prove Theorem \ref{thm:main}, however it serves as a foundational subsection and contributes to a better understanding of foliated contact topology.\\

We begin with the simpler notion of a symplectic foliation.

\begin{definition}\label{def:sympfol} Let $\SF$ be a codimension--1 foliation on a smooth manifold $M$ and $\omega \in \Omega^2(\SF)= \Lambda^2(T^* \SF)$ a 2--form. The pair $(\SF,\omega)$ is a symplectic foliation if, for every leaf $\SL$, $(\SL, i_\SL^* \omega)$ is a symplectic manifold.
\end{definition}

These objects are a generalisation of symplectic fibrations, see \cite{MS}, and they are also referred to as regular Poisson foliations, confer \cite{He}. The subtlety in the symplectic case, in contrast to the contact case, is the closedness of the possible {\it extensions } of the 2--form $\omega$. Given a symplectic foliation, we can consider the space of closed $2$--forms $\Omega$ such that $\Omega|_{T\SF} = \omega$. Then, a pair $(\SF, \Omega)$ with such an $\Omega$ will be called an extension of $(\SF, \omega)$. Symplectic foliations with a fixed extension are also referred to as $2$--calibrated structures \cite{IM2}.\\

The closedness of the extension 2--form is relevant in regard to the concept of a symplectic connection.

\begin{definition}\label{def:sympcon}
Let $(\SF,\omega)$ be a symplectic foliation. Given an extension $(\SF, \Omega)$, the symplectic connection is the distribution $\SH_\Omega = (T\SF)^{\perp \Omega}$.
\end{definition}

If a defining form $\beta$ for the foliation is chosen, a symplectic connection determines a distinguished transverse vector field $T$ defined by
$$\Omega(T)=0, \quad \beta(T)=1, $$
a notion that was already defined, for a particular subclass of symplectic foliations, in \cite{GMP}. The analogue of Lemma \ref{lem:contparallel} for symplectic foliations reads as follows.
\begin{lemma}\label{lem:sympparallel}
Let $(\SF,\omega)$ be a symplectic foliation with extension $(\SF, \Omega)$ and $\beta$ the defining $1$-form for $\SF$. Then:
\[ \SL_T \Omega = 0 \]
\end{lemma}
\begin{proof}
By the definition of $T$ and the closedness of $\Omega$, Cartan formula yields
$$\SL_T \Omega = di_X\Omega + i_X d\Omega = 0.$$
\end{proof}
Note that closedness of $\Omega$ is exactly the condition one needs to ensure that the parallel transport induced by the symplectic connection is by symplectomorphisms in the particular case of a fibration, see \cite[Lemma 6.18]{MS}. \\

Consider the space $\mathfrak{E}(\SF, \omega)$ of all extensions of $\omega$. Notice that closedness is a non--trivial condition, and $\mathfrak{E}(\SF, \omega)$ might be empty. This differs from the flexibility found in contact foliations (this difference already appears between contact and symplectic fibrations). The symplectic version of Lemma \ref{lem:affine} also holds.

\begin{lemma}\label{lem:sympcon}
$\mathfrak{E}(\SF, \omega)$ has a natural affine structure space.
\end{lemma}
\begin{proof}
Suppose that $\mathfrak{E}(\SF, \omega)$ is non--empty, fix an extension $\Omega_0$, and let $\beta$ be a defining $1$--form for $\SF$. Then the set
$$\{ \Omega \in \Omega^2(M) : \Omega \wedge \beta = \Omega_0 \wedge \beta, d\Omega = 0 \}$$
is an affine space modelled on the space of closed $2$--forms that vanish along the foliation.
\end{proof}
In order to distinguish the cases in which an extension exists (and thus there is a parallel transport by symplectomorphisms) we introduce the following notation.
\begin{definition} A symplectic foliation $(\SF, \omega)$ is a strong symplectic foliation if the space $\mathfrak{E}(\SF, \omega)$ is non--empty.
\end{definition}

A strong symplectic foliation will be abbreviated to an $s$--symplectic foliation. Observe that the assumption of a foliation being $s$--symplectic is meaningful. Indeed, many results from symplectic topology are likely to extend to this special class of foliations. For instance, approximately holomorphic techniques, see \cite{Ma1}.\\

Let us establish Moser's Lemma and Darboux's Theorem for the case of foliated symplectic topology. Their proofs are left to the reader.

\begin{lemma}[Moser Stability, see \cite{HMS}]\label{lem:moser}
Consider a foliation $\SF$ on a closed manifold $M$ and $\{\omega_t\}_{t \in [0,1]}$ a smooth family of foliated $2$--forms such that $(\SF, \omega_t)$ is a symplectic foliation for every $t$. Let $\{\Omega_t\}$ be a smooth family of associated extensions. Suppose that $[\Omega_t] \in H^2_{DR}(M)$ is constant, then there exists a global flow $\{\phi_t\}_t\in\Diff(M)$ tangent to $\SF$ and such that $\phi_t^* \omega_t =\omega_0$.\hfill$\Box$
\end{lemma}

Darboux's Theorem shall provide a local normal form for a symplectic foliation. We describe this local model in the following example.

\begin{example}\label{ex:standardSympFoliation}
The product manifold $\mathbb{R}(t) \times \mathbb{C}^n(x,y)$ can be endowed with the strong symplectic foliation $(\SF_{st},\Omega_{st})$ defined by
\[ \SF_{st} = \{t\} \times \mathbb{C}^n, \quad \Omega_{st} = \sum_{i=1}^n dx_i \wedge dy_i. \]
The standard defining form for $\SF_{st}$ is the 1--form $\beta_{st} = dt$.
\end{example}
We can state Darboux's Theorem in foliated symplectic topology.
\begin{lemma}[Darboux's Theorem]\label{lem:Darboux}
Consider a symplectic foliation $(M,\SF)$ and a point $p\in M$. Then there exists a small $\varepsilon > 0$ and an embedding $$\phi:  D^{2n}_\varepsilon \times (-\varepsilon, \varepsilon)\longrightarrow M$$ such that
$$\phi(0,0) = p, \quad  \phi^*\SF = \SF_{st}, \quad \phi^*\omega = (\Omega_{st})_{\SF_{st}}.$$
\hfill$\Box$
\end{lemma}
The strict version of Darboux's Theorem also holds.
\begin{corollary}[Strict Darboux's Theorem]\label{lem:DarbouxStrict}
Consider an s--symplectic foliation $(M,\SF,\omega)$, a point $p\in M$ and an extension $\Omega$. Then there exists a small $\varepsilon > 0$ and an embedding
$$\phi: D^{2n}_\varepsilon \times (-\varepsilon, \varepsilon)\longrightarrow M.$$
such that
$$\phi(0,0) = p, \quad  \phi^*\SF = \SF_{st}, \quad \phi^*\Omega = \Omega_{st}.$$
\hfill$\Box$
\end{corollary}

These Darboux's Theorems also hold for contact foliations (their proofs rely on the foliated Gray's stability theorem). The corresponding statements and details are left to the interested reader.\\

This Section \ref{sec:intro} has introduced the basic definitions and properties of contact and symplectic foliations. Section \ref{sec:contsymptaut} provides some constructions and examples which illustrate the richness of these geometries.

%

\section{Constructions and tautness}\label{sec:contsymptaut}

In this section we present possible constructions of foliated contact and symplectic structures. These are natural generalizations of methods used in contact and symplectic topology, although the structure of the foliation provides interesting classes of examples. These subsections can be read independently.\\

Subsection \ref{ssec:tautness} briefly recalls the definition of a taut foliation (which will be used in Sections \ref{sec:hprin} and \ref{sec:proof}). An interesting example is presented in Subsection \ref{ssec:tori}. Subsection \ref{ssec:contactElements} introduces the space of foliated contact elements and proves Proposition \ref{prop:notinjective}. Finally, Subsections \ref{ssec:contactization}, \ref{ssec:symplectization} and \ref{ssec:div} explain contactization, symplectization and foliated connected sum along divisors.

\subsection{Tautness}\label{ssec:tautness}
This is a property associated to a codimension one foliation $\SF$ on a smooth ($n+1$)--fold $V$. A foliation $\SF$ on a smooth manifold $V$ is said to be (topologically) taut if for every leaf there exists a transverse circle intersecting that leaf.\\

There is also the notion of a geometrically taut foliation: $\SF$ is geometrically taut if there exists a closed form $\tau \in \Omega^n(V)$ such that $i_\SL^* \tau$ is a volume form for every leaf $\SL$. This is equivalent to the existence of a complete non--vanishing vector field $X \in \mathfrak{X}(V)$ transverse to $\SF$ and preserving some volume form $\nu \in \Omega^{n+1}(V)$ (i.e. $\SL_X \nu =0$).\\


A taut foliation $\SF$ is geometrically taut, and the converse holds if the ambient manifold $V$ is closed, see \cite[Theorem II.20]{Su}. Observe that an $s$--symplectic foliation $(M, \SF, \omega)$ is geometrically taut, since for any choice of extension $\Omega$, we have that the closed $(2n)$--form $\Omega^n$ is a leafwise volume form. Thus, the $s$--symplectic foliation $\SF$ is taut if $M$ is closed. There are however non--taut foliations which support contact structures on closed manifolds. We will construct a family of examples in Subsection \ref{ssec:contactElements} (and this is also consequence of the Theorem \ref{thm:main}).\\

Theorem \ref{thm:main} is first proven for the case of a taut foliation $\SF$ in Section \ref{sec:proof}.

 \label{thm:taut}
\subsection{An Example}\label{ssec:tori}
Let us construct a foliated contact structure on a codimension one foliation on the 4--torus $\T^4$. There are four natural type of foliations on the 4--torus obtained by quotienting the horizontal foliation of a 4--polydisk by 3--polydisks. In coordinates $\T^4(t,x,y,z)$ their defining equations have the form
$$\beta=(p,q,r,s)\cdot (dx,dy,dz,-dt),\quad (p,q,r,s)\in\R^4.$$
The horizontal foliation of $\T^4$ by 3--tori is given by a tuple $(p,q,r,s)$ of rationally dependent numbers. In case there are two sets of rationally dependent numbers, the leaves are diffeomorphic to $\T^2\times\R$. If there are three sets of rationally dependent values, the leaves are diffeomorphic to $\S^1\times\R^2$. Finally, if the numbers are rationally independent the leaves are dense $\R^3$ in $\T^4$.\\

Let us endow such foliations with a foliated contact structure. Suppose that $s=1$ and consider the form
$$\alpha=\sin(2\pi z)dx+\cos(2\pi z)dy.$$
This is a well--defined 1--form on $\T^4$. The condition for a foliated contact structure reads
$$\alpha\wedge d\alpha\wedge\beta=(\sin(2\pi z)dx+\cos(2\pi z)dy)\wedge(2\pi\cos(2\pi z)dz\wedge dx-2\pi\sin(2\pi z)dz\wedge dy)\wedge \beta=$$
$$=(2\pi dx\wedge dy\wedge dz)\wedge(p\cdot dx+q\cdot dy+r\cdot dz-dt)=2\pi\cdot dt\wedge dx\wedge dy\wedge dz>0$$
Hence $\alpha$ defines a foliated contact structure for any 1--form $\beta$ as above. In particular, we obtain a contact foliation with (dense) tight contact $\R^3$ leaves on $\T^4$.\\

This example is of a particular interest regarding a Weinstein--type conjecture in foliated contact topology. In this last irrational case, the foliated Reeb vector field has no periodic orbits (even though the ambient manifold $\T^4$ is compact). This example must be placed in contrast to the case with overtwisted leaves, see \cite{PP}.

\subsection{Foliated Contact Elements}\label{ssec:contactElements}

Consider a pair $(W,\SF)$ given by a codimension one cooriented foliation $\SF$ on $W$ and the manifold $V=\mathbb{P}(T^*\SF)$, the projectivised cotangent bundle of $\SF$. There exists a natural projection $\pi:V\longrightarrow W$ and hence the foliation $\SF$ can be pull--backed to $V$ to a foliation $\SF_V = \pi^* \SF$. The projectivization of the cotangent bundle of a manifold has a canonical contact structure. Similarly, the foliated manifold $V$ has a natural foliated contact structure defined by
$$ (\xi_V)_p =\{v \in T_p\SF_V: p(\pi_*v) = 0 \},\quad p\in V,$$
where the point $p$ is identified with a $1$--form in $T_{\pi(p)}\SF$ (which is well--defined up to scaling).\\

In particular, the foliated contact structure restricted to a leaf coincides with the space of contact elements of the leaf. In the same vein, the sphere bundle $S=\mathbb{S}(T^*\SF)$ associated to the cotangent space of the foliation is a foliated contact manifold $(S, \SF_S, \xi_S)$ that restricts to the space of cooriented contact elements over each leaf. The foliated contact structure ($\SF_S,\xi_S)$ can also be obtained via the pullback of $(\SF_V,\xi_V)$ through the double--cover $S\longrightarrow V$.\\

Let us provide an example of two non--isotopic foliated contact structure which are homotopic as foliated almost contact structures.\\

\begin{proof}[Proof of Proposition \ref{prop:notinjective}]
 The standard foliated contact structure on the space of foliated contact elements has tight leaves. This is standard in $3$--dimensional contact topology: the space of cooriented contact elements of a (possibly open) surface is tight. The proof of Theorem \ref{thm:main} shall imply that a foliated contact structure with (at least) an overtwisted leaf can also be constructed in the same homotopy class of foliated almost contact structures. Those two structures cannot be homotopic as foliated contact structures, because Lemma \ref{lem:gray} would imply isotopy of the two contact structures restricted to any leaf.
\end{proof} 


A particular example we consider clarifying is constructed in the following proposition. 


\begin{proposition}
Consider the Reeb foliation $(\S^3,\SF)$ and its associated space of foliated contact elements $(\mathbb{S}^3 \times \mathbb{S}^1,\SF_c, \xi_c)$. There exists a homotopy of the almost contact structures that produces a contact foliation with tight and overtwisted leaves.
\end{proposition}
\begin{proof}
Fix a loop $\gamma: \mathbb{S}^1 \to \mathbb{S}^3$ transverse to the foliation $\SF$ (and thus avoiding the unique torus leaf). Its lift to the space of foliated contact elements is a transverse loop of the form
\begin{eqnarray*}
\widetilde{\gamma}: \mathbb{S}^1 & \longrightarrow & \mathbb{S}^3 \times \mathbb{S}^1 \\
\theta & \longmapsto & (\gamma(\theta), 1).
\end{eqnarray*}
Let us insert a family of Lutz twists (refer to Section \ref{sec:general} for a more detailed account of this construction) by considering an embedding
$$\Gamma:\mathbb{S}^1\times\mathbb{S}^1\longrightarrow\mathbb{S}^3\times\mathbb{S}^1$$
such that $\Gamma(\theta,1)=\widetilde{\gamma}(\theta)$ and the $\kappa$--curve $\{s\longmapsto\kappa(s)=\Gamma(\theta,s)\}$ is tangent to $\SF_c$ and transverse to $\xi_c$.\\

Then we perform a $1$--parametric family of Lutz twists along the $\theta$--family of $\kappa$--curves. The resulting contact structure certainly has overtwisted leaves. However, the leaf $T^2 \times \mathbb{S}^1$ corresponding to the lift of the unique compact leaf on $\SF$ is tight.
\end{proof}

Note that the same proof works for any foliation $(M^3, \SF)$ with a Reeb component: the space of foliated contact elements associated to $(M^3, \SF)$ possesses a foliated contact structure with tight and overtwisted leaves.

\subsection{Contactization}\label{ssec:contactization}

Let $(M,\lambda)$ be an exact symplectic manifold, the contactization $\SC(M,\lambda)$ of $(M,\lambda)$ is the manifold $M\times\R(t)$ endowed with the contact structure $\xi= \ker(\lambda-dt)$. 

\begin{definition}
Let $(M, \SF, \omega)$ be an $s$--symplectic manifold admitting an exact extension $\Omega = d\lambda$. Then $(M\times\R(t), \SF \times \mathbb{R}, \ker(\lambda - dt))$ is called the contactization $\SC(M, \SF, \omega)$ of $(M, \SF, \omega)$.
\end{definition}

In case the symplectic manifold $(M,\omega)$ is not exact there also exists a contactization (known as prequantization). Let $(\SF,\omega)$ be an $s$--symplectic foliation on $M$, with an extension $\Omega$ such that $[\Omega/(2\pi)]$ is integral and consider the principal circle bundle $L_\Omega\longrightarrow M$ associated to $\Omega$. Construct a connection 1--form $\alpha\in\Omega^1(L_\Omega)$ with curvature $\Omega$.


\begin{definition}
The foliated contact manifold $(L_\Omega, \pi^*\SF, \ker \alpha)$ is said to be the Boothby--Wang contact foliation associated to the s--symplectic foliation $(M,\SF,\omega)$ with extension $\Omega$.
\end{definition}

Note that the Boothby--Wang contact foliation over a closed base is taut because the $s$--symplectic foliation is taut.

%

\subsection{Symplectization}\label{ssec:symplectization}

Let $(V, \xi = \ker\alpha)$ a contact manifold. The symplectic manifold $(V \times \mathbb\R(t), d(e^t\alpha))$ is known as the symplectization of $(V,\xi)$.
 
\begin{definition}
Given an extended contact foliation $(V,\SF,\xi,\Theta)$ with associated pair $(\beta,\alpha)$, the symplectization of $(V, \SF, \xi,\alpha)$ is the $s$--symplectic foliation
$$(V\times\R(t), \SF \times \mathbb{R}, \omega = d(e^t\alpha)|_{\SF}).$$
\end{definition}
Notice that $\Omega = d(e^t\alpha)$ is an extension of $\omega$ (and thus $\Omega$ is exact). The $s$--symplectic foliation obtained in the construction does not depend on the particular choice of $\alpha$ and $\Theta$ (up to foliated symplectomorphism).\\

Observe that the symplectization is geometrically taut, since it is a $s$--symplectic foliation. However it is not necessarily taut, if this were the case the starting contact foliation would be taut.

\subsection{Foliated Contact Divisor Connected Sum}\label{ssec:div}

Consider a contact foliation $(V,\SF,\xi)$ on a ($2n+2$)--fold $V$ and let $(S,\SF|_S,\xi|_S)$ be a foliated contact divisor, i.e. $S$ is a $2n$--dimensional submanifold transverse to $\SF$ and $\xi$ such that $\xi|_W$ is a leafwise contact distribution.\\

Generalizing \cite[Theorem 2.5.15]{Ge}, Lemma \ref{lem:gray} implies that the tubular neighbourhood of $W$ is uniquely determined by the conformal symplectic structure of its normal bundle. Since its normal bundle is a 2--dimensional disc bundle, the foliated contact structure of its tubular neighbourhood depends only on its oriented topological type. \\

Let $(\alpha_S,\beta_S)$ be an associated pair for $(S, \SF_S, \xi_S)$ and consider its pull--back (still denoted by $(\alpha_S,\beta_S)$) to the manifold $S \times D^2_{2\varepsilon}(r,\theta)$. Suppose that the normal bundle of $S$ is trivial and thus there exists a neighborhood $\Op(S)$ and an embedding $\phi: S \times D^2_{2\varepsilon} \to\Op(S)$ such that the pull--back of the contact foliation is 
$$(S \times D^2_{2\varepsilon}, \SF_S \times D^2_{2\varepsilon}, \ker(\alpha_{\Op(S)}))\mbox{ with associated pair }
(\beta_{\Op(S)},\alpha_{\Op(S)})=(\beta_S,\alpha_S + r^2 d\theta).$$
This is the local model along the foliated contact divisor. We describe the framework in which we can perform a foliated contact connected sum along such a foliated contact divisor.\\

Let $(V_0,\SF_0,\xi_0)$ and $(V_1,\SF_1,\xi_1)$ be two contact foliations and $f_0:S\longrightarrow V_0$, $f_1:S\longrightarrow V_1$ two embeddings of $S$ as a foliated contact divisor with trivial normal bundle. There exists two neighbourhoods $\Op(S,V_0)$ and $\Op(S,V_1)$ and two embeddings
$$f_0:S\times D^2_{2\varepsilon}\longrightarrow\Op(S,V_0),\quad f_1:S\times D^2_{2\varepsilon}\longrightarrow\Op(S,V_1)$$
conforming the local model described above (and extending $f_0$ and $f_1$).\\

The gluing region is the open manifold $\mathcal{S} =  S \times (-\varepsilon^2, \varepsilon^2) \times\mathbb{S}^1$ contact foliated as
$$(\mathcal{S}, \SF_S \times (-\varepsilon^2, \varepsilon^2) \times\mathbb{S}^1, \ker(\alpha_\mathcal{S})) \mbox{ with associated pair }(\beta_\mathcal{S},\alpha_\mathcal{S}) = (\beta_S,\alpha_S + t d\theta).$$
Note the linearity in the $t$--coordinate. Define the maps
$$F_0: S \times (0, \varepsilon^2) \times\mathbb{S}^1\longrightarrow\Op(S,V_0),\quad (p, t, \theta)\longmapsto f_0(p, t^2, \theta)$$
$$F_1: S \times (-\varepsilon^2,0) \times\mathbb{S}^1\longrightarrow\Op(S,V_1),\quad (p, t, \theta)\longmapsto f_1(p, t^2, -\theta)$$
Then the topological connected sum
$$V_0 \#_S V_1 = (V_0\setminus f_0(S)) \cup_{F_0} \mathcal{S}  \cup_{F_1} (V_1\setminus f_1(S))$$
with the foliated contact models introduced above inherits a foliated contact structure. The related construction for symplectic foliations is discussed in \cite{IM2}, though it does not preserve strongness. \\

Sections \ref{sec:intro} and \ref{sec:contsymptaut} have presented the definitions, results and constructions in foliated contact and symplectic topology. The remaining Sections \ref{sec:hprin}, \ref{sec:proof} and \ref{sec:general} shall focus on the proof of Theorem \ref{thm:main}. \\


\section{h--Principle and Local Models}\label{sec:hprin}

In this Section we state the h--principles and the technical lemmas involved in the proof of Theorem \ref{thm:main}. Subsections \ref{ssec:hLemma} and \ref{ssec:eliashberg} introduce the appropriate version of the h--principles for the open and the closed case respectively. Subsection \ref{ssec:adapt} contains the topological local models and Subsection \ref{ssec:3sk} begins to construct the foliated contact structure on the 3--skeleton of $(V^4,\SF)$.

\subsection{h--Principle for Open Manifolds} \label{ssec:hLemma}

The h--principles proven by M. Gromov \cite{Gr2} include a parametric (and relative) h--principle for the existence of contact structures on open manifolds. This h--principle can also be proven using the holonomic approximation developed in \cite{EM}.\\

Let us state the exact h--principle we use in the proof of Theorem \ref{thm:main}.

\begin{theorem}\label{thm:HolonomousRel} $($\cite{EM,Gr2}$)$
Let $U$ be an open domain in a smooth manifold $V$ and $A \subset U$ a CW--complex of codimension at least $2$. Let $K$ be a compact space and $L\subset K$ a closed subset. Consider a continuous family $\{(\xi_t, J_t)\}_{t\in K}$ of almost contact structures on $V$ which are contact in $U$ for $t\in L$ and are contact near A for $t\in K$.\\

There exists a continuous deformation $\{(\xi_{t,s}, J_{t,s})\}_{s\in[0,1]}$, relative to $U\times L\cup \Op(A)\times K$ and supported in $\Op(U)$, such that $\{(\xi_{t,1}, J_{t,1})\}_{t\in K}$ is a family of contact structures on $U$.
\end{theorem}

Observe that the dependence on the parameter can be supposed to be smooth since the condition for a contact structure is open (and thus preserved by small smoothing perturbations). Theorem \ref{thm:HolonomousRel} allows us to construct a foliated contact structure in a neighborhood of the 3--skeleton, this is explained in Subsection \ref{ssec:3sk}.

\subsection{Classification of overtwisted contact structures} \label{ssec:eliashberg}
There also exists a subclass of contact closed manifolds satisfying an existence h--principle. This is the main result in \cite{El}. Let us define this class. The disk $$\{z=0,\rho\leq\pi\}\subset(\R^3(z,r,\theta),\xi_{ot}=\ker\{\cos(r)dz+r\sin(r)d\theta\})$$
with the germ of the contact structure is the standard overtwisted disk. An almost contact 3--fold $(V^3,\xi)$ is said to be overtwisted if there exists an embedded disk $D$ such that $(\Op(D),\xi)$ is contactomorphic to the standard overtwisted disk. The class of overtwisted (almost) contact manifolds is flexible, i.e. it is classified by its homotopy data. The corresponding h--principle can be stated as follows.

\begin{theorem}$($\cite{El}$)$ \label{thm:otHolonomous}
Let $U$ be an open domain in a smooth manifold $V$, $K$ a compact space and $L\subset K$ a closed subset. Consider a continuous family $\{(\xi_t, J_t)\}_{t\in K}$ of almost contact structures on $V$ which are contact on $V$ for $t\in L$, contact on $U$ for $t\in K$ and there exists an overtwisted disk $\Delta\subset V\setminus U$ for all $\{(\xi_t, J_t)\}_{t\in K}$.\\

There exists a continuous deformation $\{(\xi_{t,s}, J_{t,s})\}_{s\in[0,1]}$, relative to $V\times L\cup U\times K$ such that $\{(\xi_{t,1}, J_{t,1})\}_{t\in K}$ is a family of contact structures on $V$.
\end{theorem}

Theorem \ref{thm:otHolonomous} is used in Sections \ref{sec:proof} and \ref{sec:general} in order to extend the foliated contact structure on a neighborhood of the 3--skeleton to the interior of the 4--cells.

\subsection{Topological Local Models}\label{ssec:adapt}

Consider a pair $(V^4,\SF)$. A cell $\sigma:\Delta\longrightarrow V$ is said to be linear with respect to $\SF$ if its image is contained in the image of a trivializing foliation chart for $\SF$ and the foliation $\sigma^*\SF$ is transverse to all the facets of $\Delta$.\\

In a linear cell the height function in the foliation chart yields a function in $\Delta$ with one maximum and one minimum in two vertices and no critical points elsewhere. See Figure \ref{fig:linear}.\\

\begin{figure}[ht]
\includegraphics[scale=0.3]{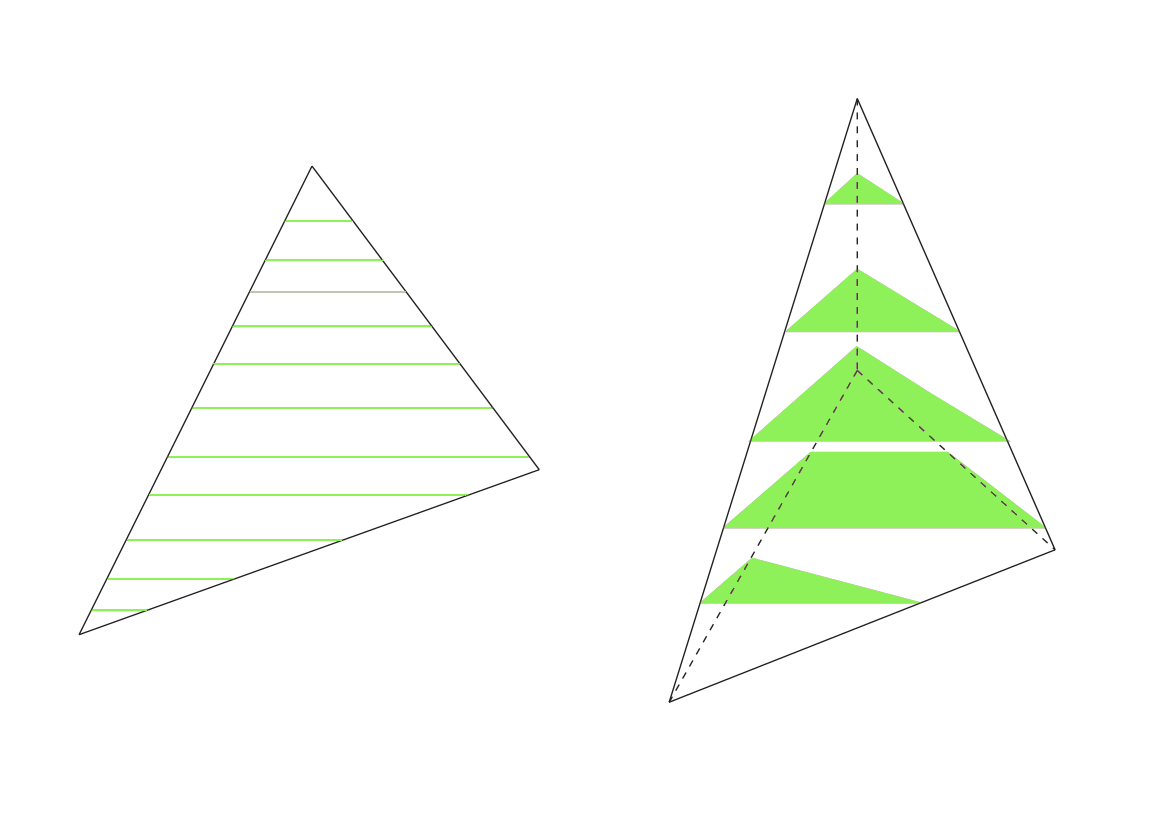}
\caption{A linear $2$--cell (left) and a linear $3$--cell (right) with their induced foliations (green).}\label{fig:linear}
\end{figure}

A triangulation $\ST$ of $V$ is adapted to the foliation $\SF$ if all its simplices are linear with respect to $\SF$. This corresponds to a distribution being in general position with respect to a triangulation. There always exists a triangulation on $V$ adapted to $\SF$, see \cite{Th}. The $i$--skeleton of $V$ with respect to this triangulation $\ST$ is denoted by $V^{(i)}$.\\

In order to use Theorems \ref{thm:HolonomousRel} and \ref{thm:otHolonomous} we require precise local models. The following lemma provides explicit models: these allow us to control the deformations that we obtain by applying the above h--principles (with the suitable parameter spaces $L\subset K$ in each occasion).\\

Consider the standard 3--ball $D^3_r\subset\mathbb{R}^3$ of radius $r$ and let us fix a sequence of equators $\mathbb{S}^0_r \subset \mathbb{S}^1_r \subset \mathbb{S}^2_r$ respectively bounding flat disks $D^1_r \subset D^2_r \subset D^3_r$. Given the interval $I=[0,1]$, the product foliation $\{\{t\}\times D^3_r\}_{t\in[0,1]}$ is denoted by $\SF_{I\times D^3_r}$. We can state the main lemma of this Subsection:

\begin{lemma} \label{lem:setupHolonomous}
Consider a triangulation $\ST$ adapted to $(V, \SF)$, an index $j \in \{0,\ldots,4\}$, a subset $G\subset V^{(j)}$ and a cell $\sigma\in V^{(j)}\setminus G$. There exists an embedding $\phi_\sigma = \phi: I \times D^3 \to V$ satisfying the following properties:\\
\begin{enumerate}
\item[a.] $\sigma \subset \im(\phi)$, $im(\phi)$ is contained in a small neighborhood of $\sigma$ and $\phi^*\SF=\SF_{I\times D^3}$.\\
\item[b.] There exist $\Op(V^{(j-1)}\cup G)\subset V$, $\Op(\partial I)\subset I$ and $\Op(\mathbb{S}^{j-2})\subset D^3$ such that\\
\begin{itemize}
\item[-] For $j=0$: $\phi^{-1}( \Op(G)) = \emptyset$,\\
\item[-] For $j=1$: $\phi^{-1}(\Op(V^{(0)}\cup G)) = \Op(\partial I) \times D^3$,\\
\item[-] For $j=2,3,4$: $\phi^{-1}(\Op(V^{(j-1)}\cup G)) = (\Op(\partial I) \times D^3) \cup (I \times \Op(\mathbb{S}^{j-2}))$.
\end{itemize}
\end{enumerate}
\end{lemma}
See Figure \ref{fig:Lemma29} for a schematic representation of the statement.

\begin{figure}[ht]
\includegraphics[scale=0.2]{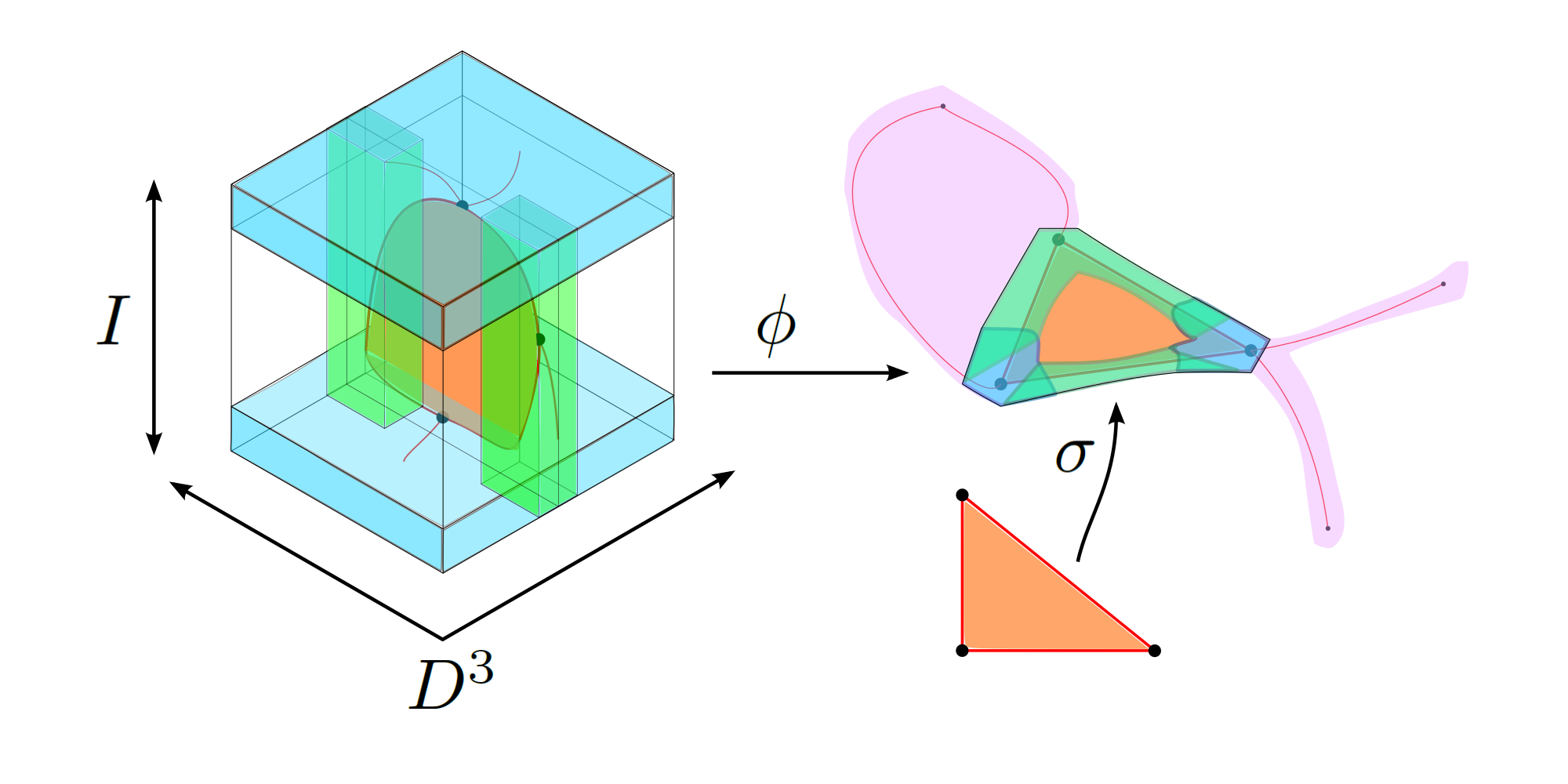}
\caption{Statement of Lemma \ref{lem:setupHolonomous} for the case $j=2$. The figure on the left depicts the local model and the one on the right a neighbourhood of its image in the manifold. The simplex $\sigma$ is depicted in orange, with edges in red. The neighbourhood $\Op(V^{(1)} \cup G)$ is colored in pink, and in this example it covers a whole simplex $\tau \in G$, to the left of $\sigma$, and two edges connected to the rightmost vertex of $\sigma$. The subsets $\Op(\partial I) \times D^3$ (blue) and $I \times \Op(\mathbb{S}^0)$ (green) cover the intersection of $\Op(V^{(1)} \cup G)$ with the image of the local model. \label{fig:Lemma29}}
\end{figure}

\begin{proof}
Consider an embedding $i:D^j\longrightarrow V$ of an open $j$--disk extending $\sigma$ such that $i^{-1}(\partial\sigma)$ is arbitrarily close to $\partial D^j$. Since the triangulation $\ST$ is adapted to $\SF$, after a small isotopy, we can suppose that $i^*\SF$ foliates the disk $D^j$ horizontally. In order to construct the embedding $\phi$ we use a normal frame along $i(D^j)$ which is contained in $\SF$. Then the exponential map (and rescaling) yields an embedding
$$\phi: D^j \times D^{4-j}\longrightarrow V,\quad\phi|_{D^j \times \{0\}}=i$$
such that $\sigma \subset \im(\phi)$, with $\im(\phi)$ arbitrarily close to $\sigma$. This map can be understood as an embedding of $I \times D^3$ and it satisfies $\phi^* \SF = \SF_{I\times D^3}$. This is a foliation chart and for $j=0$ the statement follows.\\

Suppose that $j\neq0$, let us detail the neighborhoods appearing in the statement. By construction $\phi(I \times D^3)$ is $C^0$--close to $\sigma$, hence any other simplex $\tau$ intersecting $\phi(I \times D^3)$ shares a facet with $\sigma$. In particular, there exists a neighborhood $\Op(\tau)$ such that $\Op(\tau) \cap \phi(I \times D^3) \subset \Op(\partial\sigma)$. Then $\Op(V^{(j-1)}\cup G)$ and $\Op(\partial\sigma)$ can be chosen so that $\Op(V^{(j-1)}\cup G) \cap \Op(\sigma) = \phi(\Op(\partial\sigma))$.  \\

Note now that the preimage of the boundary $\phi^{-1}(\partial\sigma)$ is arbitrarily close to $\partial D^j\times\{0\} = \partial I \times D^{j-1} \cup I \times \mathbb{S}^{j-2}$ and thus we may suppose $\phi^{-1}(\Op(\partial\sigma)) = \Op(\partial D^j \times \{0\})$. By taking very small times for the exponential map close to the minimum and maximum of the (height function of the) foliation in $\sigma$, it can be assumed that $\phi(\partial I \times D^3)$ is an arbitrarily small neighbourhood of the minimum and maximum. Then $\Op(\partial D^j)$ can be taken to be of the special form $(\partial I \times D^3) \cup (I \times \Op(\mathbb{S}^{j-2}))$, proving the claim.
\end{proof}
  
This a strictly topological result, not related to contact topology. Lemma \ref{lem:setupHolonomous} is used to apply the h--Principles in Sections \ref{ssec:hLemma} and \ref{ssec:eliashberg} with a controlled choice of spaces $L\subset K$. In particular, it is needed in Proposition \ref{prop:0123skeleton} (foliated contact structure on the 3--skeleton) and the argument on the 4--skeleton in Theorem \ref{thm:main}.\\

Theorem \ref{thm:main} holds for any pair $(V^4,\SF)$. It is however shorter (and illustrative) to consider the case of a taut foliation $\SF$. This hypothesis simplifies the argument and Theorem \ref{thm:main} is first proved in Section \ref{sec:proof} for this case. Let us state two technical brief lemmas that we use in the proof.

\begin{lemma} \label{lem:constructionGamma}
Let $\SF$ be a taut foliation on $V$. Consider a 4--cell $\sigma\in V^{(4)}$ and a map $\phi: I \times D^3 \longrightarrow V$ provided by Lemma \ref{lem:setupHolonomous}. There exists a map $\gamma_\sigma:[0,2]_{/\{0 \backsim 2\}} = \mathbb{S}^1\longrightarrow V$ transverse to $\SF$ such that:\\
\begin{itemize}
\item[-] $\gamma_\sigma(t) = \phi(t,0)$ for $t \in I$,\\

\item[-] $\gamma_\sigma(t) \in \Op(V^{(3)})$ for $t \in \mathbb{S}^1 \setminus I$.
\end{itemize}
\end{lemma}
\begin{proof}
Since the foliation $\SF$ is taut and $V$ connected, there exists a positively transverse path
$$ l:[1,2]\longrightarrow V,\quad l(1)=\phi(1,0),\quad l(2)=\phi(0,0).$$
The path $l$ can be assumed to lie in (a neighborhood of) the 3--skeleton $V^{(3)}$ up to isotopy. We can concatenate $\phi(t,0)$ for $t \in (0,1)$ and $l(t)$ for $t \in [1,2]$, and smooth the resulting map in order to obtain the required map $\gamma_\sigma$.
\end{proof}

A generic choice of maps $\phi_\sigma$ implies that: 

\begin{corollary} \label{cor:curves}
Given a finite collection $G\subset V^{(4)}$ there exists a collection $\{\gamma_\sigma\}_{\sigma\in G}$ of pairwise disjoint paths satisfying the properties of Lemma \ref{lem:constructionGamma}.
\end{corollary}

Both Lemma \ref{lem:constructionGamma} and Corollary \ref{cor:curves} are used in Section \ref{sec:proof} to conclude Theorem \ref{thm:main} in the case of a taut foliation. Their analogues for the general case of Theorem \ref{thm:main} are stated in Section \ref{sec:general}.

\subsection{3--skeleton}\label{ssec:3sk}

The argument for Theorem \ref{thm:main} is constructive on the cells of a triangulation of $V$. It begins with the lower dimensional skeleta and concludes with the construction of a foliated contact structure on the cells of $V^{(4)}$. This is a commonly used strategy for these flexibility results, for instance \cite{El}. It is often the case that the deformation required in the skeleta below the top--dimensional strata can be easily achieved.

\begin{proposition} \label{prop:0123skeleton}
Let $(\xi,J)$ be a foliated almost contact structure on $(V,\SF)$. There exists a homotopy $\{(\xi_t, J_t)\}$ of foliated almost contact structures such that $(\xi_0, J_0)=(\xi, J)$ and $(\xi_1,J_1)$ is a contact foliation on $V^{(3)}$.
\end{proposition}

\begin{proof}
Consider the 0--skeleton $V^{(0)}$. Lemma \ref{lem:setupHolonomous} gives disjoint foliation charts $\phi:I \times D^3\longrightarrow\Op(\sigma)$ near each $0$--simplex $\sigma$. The foliated almost contact structure $(\phi^*\xi_0, \phi^*J_0)$ can be considered as a 1--parametric family $(\xi_t, J_t)$ of almost contact structures in the disc $D^3$. Then we can apply Theorem \ref{thm:HolonomousRel} with $K=I$, $V=D^3$, $U=D^3_{0.5}$, $L=\emptyset$ and $A=\emptyset$ and obtain a foliated contact structure in a neighbourhood of the 0--skeleton.\\

We fix an index $j\in\{1,2,3\}$ and apply induction on the set of simplices of the $j$--skeleton $V^{(j)}$. Given a simplex $\sigma$, Lemma \ref{lem:setupHolonomous} provides an embedding $\phi:I \times D^3\longrightarrow\Op(\sigma)$ and we can consider a small $\delta>0$ and an embedding $\phi:I\times D^3_{1-\delta}\longrightarrow V$ which also conforms the properties of Lemma \ref{lem:setupHolonomous}. Then Theorem \ref{thm:HolonomousRel} applies with $K=I$, $V=D^3$, $U=D^3_{1-\delta}$, $L=\Op(\partial I)$, $A=\mathbb{S}^{j-2}_{1-\delta}$. This inductive procedure constructing the contact foliated structure (relative to the previous step) applied to the 1, 2 and 3--skeleton implies the statement.\\
\end{proof}

In the next two sections we prove the theorem. First, it is concluded for the case of a taut foliation $\SF$ in Section \ref{sec:proof}. Then Section \ref{sec:general} shortly adapts Section \ref{ssec:adapt} and proves the general result stated in Section \ref{sec:intro}.

\section{Taut Case}\label{sec:proof}

In this section we prove Theorem \ref{thm:main} in the case that the foliation $\SF$ is taut. We fix a triangulation $\ST$ of $V$ adapted to the taut foliation $\SF$. Given a foliated almost contact structure $(\xi,J)$, its associated pair is denoted by $(\beta,\alpha)$.\\

The proof consists of two steps. The first step, Proposition \ref{lem:localModelCurve}, provides an appropriate normal form for the foliated almost contact structure (or rather its defining form) in a neighborhood of a 4--cell. The second step, Proposition \ref{prop:4skeleton}, uses this local model in order to insert a foliated family of overtwisted disks.\\

In order to obtain a normal form for the defining form of the foliated almost contact structure we choose a convenient trivialization. This method has been used in the contact setup, see \cite{CP}. The following Proposition adapts the technique to the foliated framework.

\begin{proposition}\label{lem:localModelCurve}
Suppose $(\xi, J)$ is a foliated contact structure on $\Op(V^{(3)})$ and consider the set of maps $\{\gamma: \mathbb{S}^1\longrightarrow V\}$ given by Corollary \ref{cor:curves} applied to $G=V^{(4)}$. For each such $\gamma$, there exists an embedding 
$$\phi_\gamma: \mathbb{S}^1 \times D^3_R\longrightarrow V$$
for some radius $R\in\R^+$ such that $\phi(\cdot,0) = \gamma(\cdot)$ and satisfying the following properties:
\begin{enumerate}
\item[a.] $\phi^* \SF = \SF_{\mathbb{S}^1 \times D}$,
\item[b.] $\phi^{-1}(\Op(V^{(3)})$ is saturated by leaves,
\item[c.] $\phi^*\xi = \ker(\widetilde\alpha) \cap \phi^* \SF$, where $\widetilde\alpha = dz + f(t;z,r,\theta) d\theta$ and $f(t;0) = 0$.
\end{enumerate}
Also, there exists a small $\delta > 0$ such that the bounds
\begin{equation} \label{eq:boundsf}
	\partial_rf(t;r,z,\theta) > 2\delta r, \qquad  f(t;r,z,\theta) > \delta r^2, 
\end{equation}
hold for all $t$ such that $\gamma(t) \in \Op(V^{(3)})$.
\end{proposition}
\begin{proof}
Consider a path $\gamma_\sigma$ for a fixed $\gamma\in G$ and a set $\{T,X,Y,Z\}$ of commuting vector fields trivializing $\Op(\gamma_\sigma)$ such that:
\begin{itemize}
\item[-] $T$ preserves the foliation $\SF|_{\Op(\gamma_\sigma)}$ and the triple $\{X,Y,Z\}$ trivializes $T\SF|_{\Op(\gamma_\sigma)}$,
\item[-] $\alpha(Z)>0$ and $\alpha(t,0)(X)=\alpha(t,0)(Y)=0$.
\end{itemize}
This trivialization integrates to coordinates $(t,x,y,z)$ and an embedding $\psi : \mathbb{S}^1 \times D^3_{\rho}\longrightarrow V$ for some small $\rho > 0$. The pull--back of the defining form in these coordinates reads $\psi^*\alpha = h_0 dz + h_1 dx + h_2 dy$, where $h_0, h_1, h_2\in C^\infty(\mathbb{S}^1 \times D^3_\rho)$ satisfy $h_0 > 0$ and $h_1(t,0) = h_2(t,0) = 0$. Since $h_0$ is positive, we can rescale the form by $1/h_0$ and then changing to polar coordinates yields a local expression of the form $\psi^*\alpha = dz + fdr + gd\theta$, where the equality is up to a conformal factor, $f = O(r)$ and $g = O(r^2)$.\\

In order to conclude the statement we need to erase the $dr$ factor, or equivalently, to find a radial coordinate belonging to the distribution. Consider the fibration 
\begin{eqnarray*}
\{t\} \times D^3_{\rho} & \longrightarrow & D^2_{\rho} \\
(t;z,r,\theta) & \longmapsto & (r,\theta),
\end{eqnarray*}
endowed with the Ehresmann connection $\ker(\alpha) \bigcap \ker dt$. The radial vector field $\partial_r$ on the base $D^2_\rho$ lifts to $h_r = \partial_r - f\partial_z$. Let $F_{r,\theta}^t: [-\rho/2, \rho/2] \longrightarrow [-\rho, \rho]$ denote the flow of the vector field $-f\partial_z$ for a fixed value of $(t;r,\theta)$ and consider the diffeomorphism:
\begin{eqnarray*}
\Phi: \mathbb{S}^1 \times D^3_{\rho/2} & \longrightarrow & \mathbb{S}^1 \times D^3_{\rho} \\
(t;z,r,\theta) & \longmapsto & (t;F_{r,\theta}^t(z), r,\theta).
\end{eqnarray*}
The defining form $\alpha$ is expressed in these coordinates as $(\psi\circ\Phi)^* \alpha = dz + f d\theta$ for some (other) smooth function $f$. By shrinking $\Op(V^{(3)})$, the preimage $\Phi^{-1}(\Op(V^{(3)}))$ can be supposed to be saturated by leaves and we can fix the radius at some $0 < R < \rho/2$. Finally, the bounds on $f$ readily follow from $\alpha$ being a foliated contact form in $\Op(V^{(3)})$. \end{proof}

\begin{remark}
Given $\sigma \in V^{(4)}$, the maps $\phi_\sigma|_{I \times D^3_R}$ and $\phi_{\gamma_\sigma}|_{I \times D^3_R}$ do not necessarily agree. However, since the interval $I$ is contractible and $R > 0$ can be assumed to be sufficiently small so that $\im(\phi_\sigma^{-1} \circ \phi_{\gamma_\sigma})$ is disjoint from $\Op(I \times \mathbb{S}^{j-2})$, $\phi_\sigma$ can be deformed near $I \times \{0\}$ so that both embeddings agree in $I \times D^3_R$. \\
\end{remark}

The previous Proposition provides a local description of the foliated almost contact structure. We will use such model to deform the (almost) contact structure on the leaves to overtwisted almost contact structures. This is the content of the following Proposition.

\begin{proposition} \label{prop:4skeleton}
Suppose $(\xi, J)$ is a foliated contact structure on $\Op(V^{(3)})$. There exists a homotopy of foliated almost contact structures $(\xi_s, J_s)$ such that:
\begin{enumerate}
\item[a.] $(\xi_0, J_0)=(\xi,J)$ and $(\xi_1, J_1)$ is a foliated contact structure on $\Op(V^{(3)})$.
\item[b.] Given $\sigma\in V^{(4)}$, $\phi_\gamma^*\xi_1$ is foliated contact close to $I \times \{0\}$.
\item[c.] If we consider $\phi_\gamma^* \xi_1$ as a family $\{\xi_t\}_{t \in I}$ of contact structures on $D^3$, there exists a disc $\Delta \subset D^3$ which is overtwisted $\forall\xi_t$. 
\end{enumerate}
\end{proposition}
\begin{proof} There are two steps. First, find a foliated Darboux normal form in a neighborhood of a 1--parametric family of knots. Second, perform a Lutz twist along each of them.\\

Let us first deform the foliated (almost) contact structure provided by Proposition \ref{lem:localModelCurve} in order for it to be standard near the core. This procedure is done for each 4--cell $\sigma$ and their corresponding map $\gamma_\sigma$ as obtained in Corollary \ref{cor:curves}. In the local model $\phi_\gamma:\S^1\times D^3_R\longrightarrow V$ of Proposition \ref{lem:localModelCurve} the foliated defining form reads $\phi_\gamma^*\alpha=dz+fd\theta$ on each leaf. Consider a decreasing smooth cut--off function $\chi: [0,R]\longrightarrow [0,1]$ such that: 
$$\quad \chi(r) = 1 \text { in } [0,R/3]\mbox{ and }\chi(r) = 0 \text { in } [2R/3, R].$$
This cut--off is used in order to smoothly modify $f$ to the local model $\delta r^2$. Indeed, consider the function $\widetilde{f} \in C^\infty(\mathbb{S}^1 \times D^3_{R})$ defined by $\widetilde{f} = \delta r^2\cdot\chi(r) + f\cdot(1-\chi(r))$. The bounds in Proposition \ref{lem:localModelCurve} imply that
\begin{equation} \label{eq:partialf1}
\partial_r\widetilde{f}= \left( \chi'(\delta r^2 - f) \right) + \left( 2r\delta\chi + \dfrac{ \partial f}{\partial r}(1-\chi) \right)>0. 
 \end{equation}
on a neighborhood $\phi^{-1}(\Op(V^{(3)})$.\\

Since the 1--form $(\phi_\gamma)_*(dz + \widetilde f d\theta)$ agrees with $\alpha$ near the boundary of $\phi_\gamma(\S^1\times D^3_R)$, it extends to a global $1$--form $\widetilde\alpha$. Then, the leafwise defined plane field $\widetilde{\xi}= \ker(\widetilde{\alpha})$ is a foliated almost contact distribution which is, by construction, homotopic to $\xi$. This concludes the first step.\\

For a radius $\rho>0$ sufficiently small, consider the embedded torus
\begin{eqnarray*}
\eta: \mathbb{S}^1 \times \mathbb{S}^1 &  \longrightarrow & V \\
(t,\theta) & \longmapsto & \eta(t,\theta) =  \phi(t,0,\rho,\theta).
\end{eqnarray*}
It should be considered as an $\mathbb{S}^1$--invariant family of transverse loops to the (almost) contact structures on the leaves. We can now deform $\widetilde\xi$ by performing a full Lutz twist along each curve $\eta_t(\theta) = \eta(t,\theta)$, see \cite[Section 4.3]{Ge}. This yields an almost contact distribution $(\xi_1, J_1)$ such that for each $\theta\in\S^1$, the almost contact structure $\xi_1$ has a $1$--parametric family of overtwisted discs $\{\Delta_t\}_{t\in\S^1}$ centered at the points $\phi(t,0,\rho,\theta)$. The resulting foliated almost contact structure satisfies the properties of the statement.
\end{proof}

The results from Section \ref{sec:hprin} and the previous Proposition are enough to conclude Theorem \ref{thm:main} in the case that $\SF$ is a taut foliation.\\

\begin{proof}[Proof of Theorem \ref{thm:main} (Taut Case)]
We first apply Proposition \ref{prop:0123skeleton} to the foliated almost contact structure $(\xi, J)$ in order to obtain a foliated almost contact structure which is foliated contact on $\Op(V^{(3)})$.  This foliated almost contact structure satisfies the hypotheses of Proposition \ref{prop:4skeleton} and thus there exists a deformation to a foliated almost contact structure which is still contact near $\Op(V^{(3)})$ and each leafwise almost contact structure on the 4--cells has an overtwisted disk. Hence, we can apply Proposition \ref{thm:otHolonomous} to each 4--cell with $K=I$, $V=D^3$, $U = \Op(\mathbb{S}^2)$, $L = \partial I$ and the overtwisted disk $\Delta$ (where the neighborhoods are obtained using Proposition \ref{prop:4skeleton}). This yields a deformation of foliated almost contact structures to a foliated contact structure $(\xi_1, J_1)$.
\end{proof}

This concludes Theorem \ref{thm:main} for a taut foliation. The proof does not directly apply to the case of a general foliation $\SF$, however the strategy can be modified. This is explained in the subsequent section. \\

\section{Proof of the general case} \label{sec:general}

\subsection{Proof of Theorem \ref{thm:main}} \label{ssec:thm1}

\subsubsection{Preliminaries}

In the case that $\SF$ is a taut foliation, the argument in Section \ref{sec:proof} applies. In order to adapt the proof for a general foliation we introduce the notion of a vanishing Lutz twist. We fix a triangulation $\ST$ of $V$ adapted to the foliation $\SF$ and denote by $(\beta,\alpha)$ a pair associated to the foliated almost contact structure $(\xi,J)$.\\

Note that Proposition \ref{prop:0123skeleton} applies to any almost contact foliation. Hence, we suppose that $\xi$ is a foliated contact structure in $\Op(V^{(3)})$. We first state two lemmas as in Section \ref{ssec:adapt} and then introduce the vanishing family of Lutz twists in Subsection \ref{ssec:vanish}. This allows us to conclude Theorem \ref{thm:main} for any foliated almost contact structure on any codimension--1 foliated 4--fold $(V,\SF)$.\\

The arguments for the following two results are closely related to the proofs of Lemma \ref{lem:constructionGamma} and Proposition \ref{lem:localModelCurve}. We include the corresponding statements and leave the details of the proofs to the reader. $\widetilde\Op(X)$ denotes an arbitrarily small neighbourhood of the set $X$ that is chosen to be relatively compact within some $\Op(X)$. 

\begin{lemma}\label{lem:curvesGeneral}
Consider a 4--cell $\sigma\in V^{(4)}$ and the map $\phi_\sigma$ (as in Lemma \ref{lem:setupHolonomous}). There exist a sequence of open intervals $I \subsetneq I_1 \subsetneq I_2 \subsetneq I_3$ and a map $\gamma_\sigma: I_3 \longrightarrow V$ transverse to $\SF$ such that:
\begin{itemize}
\item[-] $\gamma_\sigma(t) = \phi(t;0)$ for $t \in I$,
\item[-] $\gamma_\sigma(I_1 \setminus I) \subset \widetilde\Op(V^{(3)})$, $\gamma_\sigma(I_2 \setminus I_1) \subset \Op(V^{(3)}) \setminus \widetilde\Op(V^{(3)})$ and $\gamma_\sigma(I_3 \setminus I_2) \subset V \setminus \Op(V^{(3)})$ .\\
\end{itemize}
Also, $\gamma_\sigma(I_3)$ and $\gamma_\tau(I_3)$ are disjoint if $\tau\in V^{(4)}$ is different from $\sigma$.\hfill$\Box$ \\
\end{lemma}

\begin{lemma}\label{lem:normGeneral}
Consider a map $\gamma_\sigma:I_3\longrightarrow V$ (provided by Lemma \ref{lem:curvesGeneral}). There exists an embedding 
$$\kappa_\gamma: I_3 \times D^3_R \to V$$
for some radius $R\in\R^+$ such that $\kappa(\cdot;0)=\gamma(\cdot)$ and conforming the following properties:
\begin{itemize}
\item[a.] $\kappa_\gamma^* \SF = \SF_{I_3 \times D}$,
\item[b.] $\kappa_\gamma^{-1}(\Op(V^{(3)}))$ and $\kappa_\gamma^{-1}(\widetilde\Op(V^{(3)}))$ are saturated by leaves,
\item[c.] the map $\phi_\sigma|_{I \times D^3_R}$ can be perturbed so that it agrees with $\kappa_\gamma|_{I \times D^3_R}$.
\end{itemize}
Also, there exists an almost contact structure $(\widetilde\xi, \widetilde J)$ homotopic to $(\xi, J)$ which is foliated contact on $\Op(V^{(3)})$ such that:
$$\kappa_\gamma^*{\widetilde\xi} = \ker(\widetilde\alpha)\cap\kappa_\gamma^* \SF \text{ where } \widetilde\alpha = dz + r^2d\theta,\quad\forall t\in I_2.$$
\hfill$\Box$ \\
\end{lemma}
%

\subsubsection{Vanishing Lutz twist}\label{ssec:vanish}


In the case of a general foliation $\SF$ we cannot insert a transverse (circle) family of overtwisted disks intersecting every leaf. We solve this by considering a transverse interval assigned to each 4--cell and performing a {\it compactly supported} or {\it vanishing} Lutz twist. Let us define this.\\

The almost contact structure $\widetilde\xi$ obtained in Lemma \ref{lem:normGeneral} is foliated contact and $t$--invariant in $\kappa_\gamma(I_2 \times D^3)$ and also leafwise homotopic to $\xi$. Denote by $\xi^0$ the leafwise contact structure in each $D^3$ and consider a homotopy $\{\xi^s\}$ of plane fields such that $\xi^1$ is the result of performing a Lutz twist to $\xi^0$.\\

Consider a cut--off smooth function $\chi:I_2\longrightarrow[0,1]$ such that:
\begin{itemize}
\item[-] $\chi(t) = 1$ for $t \in I_1$ and $\chi(t) = 0$ for $t$ near $\partial I_2$,
\item[-] $\chi(t)$ is monotone in the two components of $I_2 \setminus I_1$. 
\end{itemize}

\begin{definition}[Vanishing Family of Lutz twists] \label{def:vanish}
The distribution $\xi^{\chi(t)} \cap \ker(dt)$ defined in $I_2\times D^3_R$ is called a vanishing family of Lutz twists along the segment $I_2\times\{0\}$.
\end{definition}

The concept of a vanishing Lutz twist allows us to conclude Theorem \ref{thm:main}.\\

\begin{proof}[Proof of Theorem \ref{thm:main}]
We first apply Proposition \ref{prop:0123skeleton} to the foliated almost contact structure $(\xi, J)$ and obtain a foliated almost contact structure which is foliated contact on $\Op(V^{(3)})$. Consider the trivializing local model provided by Lemmas \ref{lem:curvesGeneral} and \ref{lem:normGeneral} and perform a vanishing family of Lutz twists along each of the intervals $\kappa_\gamma$, for all $\gamma\in V^{(4)}$ . Then Proposition \ref{thm:otHolonomous} applied to each 4--cell relative to the 3--skeleton gives a deformation of the foliated almost contact structures to a foliated contact structure $(\xi_1, J_1)$. 
\end{proof}
%

\subsection{Proof of Corollary \ref{cor:main}} We describe the proof of Theorem \ref{thm:main} for continuous parametric families of almost contact structures $\{(\xi_t, J_t) \}_{t \in P}$, for an arbitrary compact set of parameters $P$. The steps carried out in the previous subsection admit a parametric version, let us briefly discuss it.

\begin{itemize}
\item[1.] Since the parameter space $P$ is compact, the neighbourhoods $\Op(V^{(j-1)}\cup G)$ (introduced in Lemma \ref{lem:setupHolonomous}) and $\Op(V^{(j)})$, in which the structures are already foliated contact, can be assumed not to depend on $t \in P$. Therefore, the results contained in Subsection \ref{ssec:adapt} refer only to the fixed foliation $\SF$ and thus the addition of parameters is trivial. The same reasoning applies to Lemma \ref{lem:curvesGeneral}.

\item[2.] In the zero skeleton we proceed as in Proposition \ref{prop:0123skeleton}, by setting $V=D^3$, $U=D^3_{0.5}$, $K = I \times P$, $L=\emptyset$, and $A=\emptyset$, and applying Theorem \ref{thm:HolonomousRel}. For the case $j \in \{1,2,3\}$, define $V=D^3$, $U=D^3_{1-\delta}$, $K = I \times P$, $L=\Op(\partial I) \times P$ and $A=\mathbb{S}^{j-2}_{1-\delta}$.

\item[3.] The deformation of $\xi$ to $\widetilde\xi$ provided by Lemma \ref{lem:normGeneral}, to which one applies the vanishing Lutz twist, can be reproduced parametrically. Indeed, the curves $\gamma$ around which the deformation is done do not depend on $t \in P$ and following the proof shows that a $P$--parametric family of embeddings $\kappa_{\gamma,t}: I_3 \times D^3_R \to V$ (as in Lemma \ref{lem:normGeneral}) can be constructed so that $R$ does not depend on $t$ and the pullback of the deformation $\widetilde\xi_t$ of $\xi_t$ is given by the kernel of $dz + r^2 d\theta$.

\item[4.] Then, the vanishing Lutz twist can be produced parametrically along a fixed family of intervals, depending on $t \in P$. In each local model $I \times D^3$ this yields an overtwisted disc $\Delta_t \subset D^3$. Then, an extension of Theorem \ref{thm:otHolonomous} for the case of a parametric family of overtwisted disks yields the result considering $K = I \times P$, $V=D^3$, $U = \Op(\mathbb{S}^2)$, $L = \partial I \times P$ and the overtwisted disks $\Delta_t$. \\
\end{itemize}

Let us mention some possible extensions of Theorem \ref{thm:main}. First, the hypotheses on orientability and coorientability can be weakened. The same argument as in the orientable and coorientable case holds for the non--orientable and non--coorientable cases by taking double covers appropiately. Second, Theorem \ref{thm:main} also holds for open 4--folds. Finally, we believe that the argument can be modified to apply to the case of foliations of higher codimension (the essential step being an appropriate construction of the vanishing family of Lutz twists).




\end{document}